\documentclass[11pt,reqno]{amsart}
\usepackage[active]{srcltx} 
\usepackage{latexsym}
\usepackage{amssymb,amsmath,amsfonts}
\usepackage[ansinew]{inputenc}
\usepackage{graphicx}
\usepackage{color}
\usepackage{url}
\usepackage[colorlinks]{hyperref}

\numberwithin{equation}{section}
\newtheorem{thm}{Theorem}[section]
\newtheorem{cor}[thm]{Corollary}

\newtheorem{prop}[thm]{Proposition}
\newtheorem{ex}{Example}

\theoremstyle{definition}

\theoremstyle{remark}

\usepackage{fancyhdr}
\usepackage{graphicx}
\begin{document}

\title{Necessary and sufficient conditions for copositive tensors}%
 \thanks{Email: songyisheng1@gmail.com (Song); Liqun.Qi@polyu.edu.hk (Qi).}
\thanks{This work was supported by the Hong Kong Research Grant Council (Grant No. PolyU 501808, 501909, 502510, 502111) and the first author was supported partly by the National Natural Science Foundation of P.R. China (Grant No. 11071279, 11171094, 11271112)  and by the Research Projects of Science and Technology Department of Henan Province(Grant No. 122300410414 ).}
 \maketitle
 \begin{center}{Yisheng Song$^{1,2}$ and Liqun Qi$^1$}\\\vskip 2mm
 1. Department of Applied Mathematics, The Hong Kong Polytechnic University, Hung Hom, Kowloon, Hong Kong\\\vskip 2mm
 2. College of Mathematics and Information Science,
Henan Normal University, XinXiang HeNan,  P.R. China, 453007.\end{center}
 %
\vskip 4mm
\begin{quote}{\bf Abstract.}\
In this paper, it is proved that a symmetric tensor is (strictly)
copositive if and only if each of its principal sub-tensors has no
(non-positive) negative $H^{++}$-eigenvalue. Necessary and
sufficient conditions for (strict) copositivity of a symmetric
tensor are also given in terms of $Z^{++}$-eigenvalues of the
principal sub-tensors of that tensor. This presents a method for
testing (strict) copositivity of a symmetric tensor by means of
lower dimensional tensors. Also an equivalent definition of strictly
copositive tensors is given on the entire space
$\mathbb{R}^n$.\vskip 2mm

\vskip 2mm

\end{quote}
\vskip 8mm
\pagestyle{fancy} \fancyhead{} \fancyhead[EC]{ Yisheng Song and Liqun Qi}
\fancyhead[EL,OR]{\thepage} \fancyhead[OC]{The necessary and sufficient conditions of copositive tensors} \fancyfoot{}

\section{\bf Introduction}\label{}
An $m$-order $n$-dimensional tensor $\mathcal{A}$ consists of $n^m$
entries in the real field $\mathbb{R}$:
$$\mathcal{A} = (a_{i_1\cdots i_m}),\ \ \ \ \  a_{i_1\cdots i_m} \in \mathbb{R},\ \  1 \leq i_1,i_2,\cdots,i_m\leq n.$$
For $x = (x_1, x_2,\cdots, x_n)^T\in \mathbb{R}^n$ (or $\mathbb{C}^n$),  $\mathcal{A}x^{m-1}$ is a vector in $\mathbb{R}^n$ (or $\mathbb{C}^n$) with its ith component defined by
\begin{equation}\label{eq:11}(\mathcal{A}x^{m-1})_i=\sum_{i_2,\cdots,i_m=1}^na_{ii_2\cdots i_m}x_{i_2}\cdots x_{i_m}.\end{equation}
Then $x^T(\mathcal{A}x^{m-1})$ is a homogeneous polynomial, denoted
as $\mathcal{A}x^m$, i.e.,
\begin{equation}\label{eq:12}\mathcal{A}x^m=x^T(\mathcal{A}x^{m-1})=\sum_{i_1,i_2,\cdots,i_m=1}^na_{i_1i_2\cdots
i_m}x_{i_1}x_{i_2}\cdots
x_{i_m}\index{$\mathcal{A}x^m$},\end{equation} where $x^T$ is the
transposition of $x$.  An $m-$order $n$-dimensional tensor
$\mathcal{A}$ is called {\em nonnegative} ({\em positive}) if
$a_{i_1i_2\cdots i_m}\geq 0$ ($a_{i_1i_2\cdots i_m}>0$) for all
$i_1,i_2,\cdots, i_m$.  An $m$-order $n$-dimensional tensor
$\mathcal{A}$ is said to be {\em symmetric} if its entries
$a_{i_1\cdots i_m}$ are invariant for any permutation of the
indices. It is obvious that each $m$-order $n$-dimensional symmetric
tensor $\mathcal{A}$ defines a homogeneous polynomial
$\mathcal{A}x^m$ of degree $m$ with $n$ variables and vice versa.

 As a natural extension of the counterparts for symmetric matrices,  the
concepts of eigenvalues and eigenvectors  were introduced by Qi
\cite{LQ1} for higher order symmetric tensors, and the existence of
the eigenvalues and eigenvectors and their practical applications in
determining positive definiteness of an even degree multivariate
form were also studied by Qi \cite{LQ1}.  Lim \cite{LL}
independently introduced this notion and proved the existence of the
maximum and minimum eigenvalues using a variational approach.  Qi
\cite{LQ1, LQ2} extended some nice properties of matrices to higher
order tensors. Qi \cite{LQ2, LQ3} defined $E$-eigenvalues and the
$E$-characteristic polynomial of a tensor, and proved that an
$E$-eigenvalue of a tensor is a root of the $E-$characteristic
polynomial. Subsequently, these topics attract attention of many
mathematicians from different disciplines. For diverse studies and
applications on these topics, see  Chang \cite{C09}, Chang, Pearson
and Zhang \cite{CPT1},
Chang, Pearson and Zhang \cite{CPT},  Hu, Huang and Qi \cite{HHQ}, Hu and Qi \cite{HQ}, Ni, Qi, Wang and Wang \cite{NQWW}, Ng, Qi and Zhou \cite{NQZ},  Song and Qi \cite{SQ13,SQ10}, Yang and Yang \cite{YY10,YY11}, Zhang \cite{TZ}, Zhang and  Qi \cite{ZQ}, Zhang, Qi and Xu \cite{ZQX} and references cited therein.\\

 For a vector $x\in \mathbb{R}^n$, $x\geq 0$ ($x> 0$) means  that $x_i\geq0$ ($x_i>0$), $i=1,2,\cdots,n$. A real symmetric
matrix $A$ is said to be (i) {\em copositive} if $x\geq 0$ implies
$x^TAx\geq0$; (ii) {\em strictly copositive} if $x\geq 0$ and
$x\neq0$ implies $x^TAx>0$. This concept  is one of the most
important concept in applied mathematics and graph theory, which was
introduced by Motzkin \cite{TSM} in 1952. In the literature, there
are extensive discussions on such matrices. For example, Haynsworth
and Hoffman \cite{HH} showed the Perron-Frobenius property of a
copositive matrix; Martin \cite{DM} studied the properties of
copositlve matrices by means of definiteness of quadratic forms
subject to homogeneous linear inequality constraints;
V${\ddot{a}}$li${\dot{a}}$ho \cite{HV} developed some finite
criteria for (strictly) copositive matrices by searching its
principal submatrices; Ping and  Yu \cite{PY} obtained necessary and
sufficient conditions for copositive matrices of order four; Kaplan
\cite{WK} presented necessary and sufficient conditions for a
symmetric matrix to be (strictly) copositive by using eigenvalues
and eigenvectors of the principal submatrices of that matrix.

\begin{thm}[Kaplan  \cite{WK}]  Let $A$ be a symmetric matrix. Then $A$ is (strictly) copositive
if and only if every principal submatrix $B$ of $A$ has no eigenvector $v > 0$ with
associated eigenvalue $(\lambda\leq0) \lambda<0$.
\end{thm}

One of the most important motivation for studying copositive
matrices is that a large class of mixed-binary quadratic programs
can be formulated as copositive programs \cite{B09} where a linear
function is minimized over a linearly constrained subset of the cone
of completely positive matrices, which is the dual cone of the
copositive matrix cone. More recently, this equivalence has been
extended to general nonconvex quadratically constrained quadratic
program whose feasible region is nonempty and bounded \cite{BD11}.
Also there is a nice survey on copositive matrices and their
applications \cite{B12}. It is interesting to see to what extent
these results can be extend to the tensor. Recently, Qi \cite{LQ5}
extended the concept of copositive matrices to tensors and found its
many nice properties as copositive matrices. Suppose that a tensor
$\mathcal{A}$ is a real symmetric tensor of order $m$ and dimension
$n$. $\mathcal{A}$ is said to be \begin{itemize}
\item[(i)] {\em copositive } if $\mathcal{A}x^m\geq0$ for all $x\in \mathbb{R}^n_+$;
\item[(ii)] {\em strictly copositive} if  $\mathcal{A}x^m>0$ for all $x\in \mathbb{R}^n_+\setminus\{0\}$.\end{itemize}

  A matrix is a $2$-order tensor, so
 it is very interesting to try and establish similar results for tensors as a parallel theory for matrices.  The concept of principal sub-tensors of a symmetric tensor
 was introduced by Qi \cite{LQ1} for studying positive semidefiniteness of that tensor when the order of that tensor is even. Now we study the (strict) copositivity of
 a symmetric tensor $\mathcal{A}$ with the aid of  the principal sub-tensors of $\mathcal{A}$.\\

In this paper,  we will give an equivalent definition of  (strict)
copositivity of a symmetric tensor   on the entire  space
$\mathbb{R}^n$.  It is showed that a symmetric tensor $\mathcal{A}$
is (strictly) copositive if and only if every principal sub-tensor
of $\mathcal{A}$ has no (non-positive) negative
$H^{++}$-eigenvalues, i.e., every principal  sub-tensor of
$\mathcal{A}$ has no eigenvector $v>0$ with associated
$H$-eigenvalue ($\lambda\leq0$) $\lambda<0$.  The same conclusions
still hold for using the $Z^{++}$-eigenvalue instead of
$H^{++}$-eigenvalue.  Applying these results, we can test the
copositivity of a symmetric tensor by means of lower dimensional
tensors.

 \section{\bf Preliminaries and Basic facts}

 In the sequel, we shall denote the transposition of a vector $x$ by $x^T$. Throughout this paper, let $\mathbb{R}^n_{+}=\{x\in \mathbb{R}^n;x\geq0\}$, $\mathbb{R}^n_{-}=\{x\in \mathbb{R}^n;x\leq0\}$, and $\mathbb{R}^n_{++}=\{x\in \mathbb{R}^n;x>0\}$, and $e=(1,1,\cdots,1)^T$. Denote by $e^{(i)}$ the
ith unit vector in $\mathbb{R}^n$, i.e., $e^{(i)}_j=1$ if $i=j$ and  $e^{(i)}_j=0$ if $i\neq j$, for $i,j\in\{1,2,\cdots,n\}$.\\

 Let $\mathcal{A}$ be an $m$-order $n$-dimensional tensor.
A number $\lambda\in \mathbb{C}$ is called an {\em eigenvalue of $\mathcal{A}$}, if it and a nonzero vector $x\in \mathbb{C}^n\setminus\{0\}$ are solutions of the following systems of equations:
\begin{equation}\label{eq:22}\mathcal{A}x^{m-1}=\lambda x^{[m-1]}, \end{equation}
where $x^{[m-1]}=(x_1^{m-1},\cdots , x_n^{m-1})^T$, and call $x$ an {\em eigenvector} of $\mathcal{A}$ associated with the eigenvalue $\lambda$. We call such an eigenvalue {\em $H$-eigenvalue} if it is real and has a real eigenvector $x$, and call such a real eigenvector $x$ an {\em H-eigenvector}.

These concepts  were first introduced by Qi \cite{LQ1} for higher
order symmetric tensors. Lim \cite{LL} independently introduced this
notion.
Qi \cite{LQ1, LQ2} extended some nice properties of symmetric matrices to higher order symmetric tensors. The Perron-Frobenius theorem of nonnegative matrices had been generalized to higher order nonnegative tensors under various conditions by Chang, Pearson and Zhang \cite{CPT1},  Hu, Huang and Qi \cite{HHQ}, Yang and Yang \cite{YY10,YY11}, Zhang \cite{TZ} and others.\\

For an $m$-order $n$-dimensional tensor $\mathcal{A}$, we say a number $\mu\in \mathbb{C}$ is an {\em $E$-eigenvalue of $\mathcal{A}$} and a nonzero
vector $x\in \mathbb{C}^n\setminus\{\theta\}$ is an {\em $E$-eigenvector} of $\mathcal{A}$ associated with the $E$-eigenvalue $\mu$, if they are solutions of the following systems of equations:
\begin{equation}\label{eq:23}\begin{cases}\mathcal{A}x^{m-1}=\mu x\\
 x^Tx=1. \end{cases}\end{equation}
 If $x$ is real, then $\mu$ is also real. In this case, $\mu$ and $x$ are called a {\em $Z$-eigenvalue} of  $\mathcal{A}$ and
a {\em Z-eigenvector} of  $\mathcal{A}$ associated with the Z-eigenvalue $\mu$, respectively.

These concepts about $E$-eigenvalue were first introduced  by Qi \cite{LQ1,LQ2} for studying the properties of a higher order tensor. Qi \cite{LQ3} defined the $E$-characteristic polynomial of a tensor $\mathcal{A}$, and showed that if $\mathcal{A}$ is regular, then  a complex number is an E-eigenvalue of $\mathcal{A}$ if and only if it is a root of the E-characteristic polynomial.\\

Recently, Qi \cite{LQ4} introduced and used the following concepts
for studying the properties of hypergraphs. An $H$-eigenvalue
$\lambda$ of $\mathcal{A}$ is said to be (i) an {\em
$H^+$-eigenvalue of $\mathcal{A}$}, if its $H$-eigenvector $x\in
\mathbb{R}^n_+$;(ii) an {\em  $H^{++}$-eigenvalue of $\mathcal{A}$},
if its $H$-eigenvector $x\in \mathbb{R}^n_{++}$.
    Similarly, we introduce the concepts of $Z^+$-eigenvalue and $Z^{++}$-eigenvalue.  An $Z$-eigenvalue $\mu$ of $\mathcal{A}$ is said to be (a) a {\em $Z^+$-eigenvalue of $\mathcal{A}$}, if its $Z$-eigenvector $x\in \mathbb{R}^n_+$; (b) a {\em  $Z^{++}$-eigenvalue of $\mathcal{A}$}, if its $Z$-eigenvector $x\in \mathbb{R}^n_{++}$.\\

\section{\bf Strictly Copositive  Tensors with respect to entire space $\mathbb{R}^n$}

Let $\|\cdot\|$ denote any norm on $\mathbb{R}^n$. For $x=(x_1,x_2,\cdots,x_n)^T$, let $$x^+=(x_1^+,x_2^+,\cdots,x_n^+)^T\mbox{ and }x^-=(x_1^-,x_2^-,\cdots,x_n^-)^T,$$ here $x_i^+=\max\{x_i,0\}$ and $x_i^-=\max\{-x_i,0\}$ for $i=1,2,\ldots,n$. Clearly, $x^+\geq0$, $x^-\geq0$, $|x_i|=x_i^++x_i^-$,  and $x=x^+-x^-$. Now we give the equivalent definition of  (strict) copositivity of a symmetric tensor in the sense of any norm on $\mathbb{R}^n$.

\begin{prop} \label{pr:41} Let $\mathcal{A}$ be a symmetric tensor of order $m$ and dimension $n$. Then we have
\begin{itemize}
\item[(i)] $\mathcal{A}$ is copositive if and only if  $\mathcal{A}x^m\geq0$ for all $x\in \mathbb{R}^n_+$ with $\|x\|=1$;
\item[(ii)] $\mathcal{A}$ is strictly copositive if and only if $\mathcal{A}x^m>0$ for all $x\in \mathbb{R}^n_+$ with $\|x\|=1$;
\item[(iii)] $\mathcal{A}$ is strictly copositive if and only if $\mathcal{A}$ is copositive and the fact that $\mathcal{A}x^m=0$ for $x\in \mathbb{R}^n_+$  implies  $x=0$.
\end{itemize}
 \end{prop}

 \begin{proof}
 (i) When $\mathcal{A}$ is copositive, the conclusion is obvious. Conversely, take $x\in \mathbb{R}^n_+$. If $\|x\|=0$, then it follows that $x=0$, and hence $\mathcal{A}x^m=0$. If $\|x\|>0$, then let $y=\frac{x}{\|x\|}$. We have $\|y\|=1$ and $x=\|x\|y$, and so
 $$\mathcal{A}x^m=\mathcal{A}(\|x\|y)^m=\|x\|^m\mathcal{A}y^m\geq0.$$
 Therefore, $\mathcal{A}x^m\geq0$ for all $x\in \mathbb{R}^n_+$, as required.

 Similarly, (ii) is easily proved.

 (iii) Let $\mathcal{A}$ be strictly copositive. Clearly, $\mathcal{A}$ is copositive. Suppose there exists $x_0\in \mathbb{R}^n_+$ and $x_0\neq0$ such that $\mathcal{A}x_0^m=0$, which contradicts the strict copositivity of  $\mathcal{A}$.  Conversely, if $x\neq0$ and $x\in \mathbb{R}^n_+$, then $\mathcal{A}x^m\neq0$. Since $\mathcal{A}$ is copositive, $\mathcal{A}x^m>0$. The conclusion follows.
 \end{proof}

Next  we present the necessary and sufficient conditions of strict copositivity of a symmetric tensor on entire space  $\mathbb{R}^n$.

 \begin{thm} \label{th:42} Let $\mathcal{A}$ be a symmetric tensor of order $m$ and dimension $n$. Then $\mathcal{A}$ is strictly copositive if and only if there is a real number $\gamma\geq0$ such that  \begin{equation} \label{eq:41}\mathcal{A}x^m+\gamma\|x^-\|^m>0, \mbox{ for all }x\in \mathbb{R}^n\setminus\{0\}.\end{equation}
  \end{thm}

 \begin{proof}
 Let $\mathcal{A}$ be strictly copositive. Suppose that there is no $\gamma\geq0$ such that the inequality (\ref{eq:41}) holds, i.e., for any real number $\gamma\geq0$, there exists an $x^{(\gamma)}\in\mathbb{R}^n\setminus\{0\}$ such that $$\mathcal{A}(x^{(\gamma)})^m+\gamma\|(x^{(\gamma)})^-\|^m\leq0.$$
 In particular, for any positive integer $k$ (taking $\gamma=k$), there exists $x^{(k)}\in\mathbb{R}^n\setminus\{0\}$ such that    $$\mathcal{A}(x^{(k)})^m+k\|(x^{(k)})^-\|^m\leq0.$$
 Clearly, $\|x^{(k)}\|>0$. Let $y^{(k)}=\frac{x^{(k)}}{\|x^{(k)}\|}.$ Then we have
 \begin{equation} \label{eq:43}\mathcal{A}(y^{(k)})^m+k\|(y^{(k)})^-\|^m\leq0\mbox{ for all positive integer } k,\end{equation}
 and hence, \begin{equation} \label{eq:44}\frac{\mathcal{A}(y^{(k)})^m}k+\|(y^{(k)})^-\|^m\leq0 \mbox{ for all positive integer } k.\end{equation}
  Since $\|y^{(k)}\|=1$ for all positive integer $k$, we may assume that the sequence $\{y^{(k)}\}$ strongly converges to some vector $y\in \mathbb{R}^n\setminus\{0\}$ with $\|y\|=1$ (extracting a subsequence if necessary).
 Let $k\to\infty$ in (\ref{eq:44}). Then we have $\|y^-\|=0$, and so $y\in \mathbb{R}^n_+\setminus\{0\}$. It follows from the strict copositivity of $\mathcal{A}$ that $$\mathcal{A}y^m>0.$$
 Since $\lim\limits_{k\to\infty}\mathcal{A}(y^{(k)})^m=\mathcal{A}y^m$, there exists a positive integer $N$ such that $$\mathcal{A}(y^{(k)})^m>0\mbox{ for all }k>N.$$
 This yields a contradiction of (\ref{eq:43}). Thus there is $\gamma\geq0$ such that the inequality (\ref{eq:41}) holds.

  Conversely, take $x\in \mathbb{R}^n_+\setminus\{0\}$. Obviously, $\|x^-\|=0$. Therefore, it follows from the inequality (\ref{eq:41}) that
 $$\mathcal{A}x^m>0\mbox{ for all }x\in \mathbb{R}^n_+\setminus\{0\}.$$
 Therefore, $\mathcal{A}$ is strictly copositive, as required.
 \end{proof}

 When $\mathcal{A}$ is a symmetric tensor of even order,  $x^-$ may be replaced by $x^+$ in Theorem \ref{th:42}.

\begin{thm} \label{th:43} Let $\mathcal{A}$ be a symmetric tensor of order $m$ and dimension $n$.   If  $m$ is an even number,  then $\mathcal{A}$ is strictly copositive if and only if there is a real number $\gamma\geq0$ such that \begin{equation} \label{eq:42}\mathcal{A}x^m+\gamma\|x^+\|^m>0, \mbox{ for all }x\in \mathbb{R}^n\setminus\{0\}.\end{equation}
 \end{thm}
\begin{proof}
 Let $\mathcal{A}$ be strictly copositive. Since $m$ is an even number,  $$\mathcal{A}(-x)^m=\mathcal{A}x^m.$$ Suppose that for any positive integer $k$, there exists $x^{(k)}\in\mathbb{R}^n\setminus\{0\}$ such that $$\mathcal{A}(x^{(k)})^m+k\|(x^{(k)})^+\|^m\leq0.$$
Take $y^{(k)}=\frac{x^{(k)}}{\|x^{(k)}\|}.$ Then we have
  \begin{equation} \label{eq:45}\frac{\mathcal{A}(y^{(k)})^m}k+\|(y^{(k)})^+\|^m\leq0 \mbox{ for all positive integer } k.\end{equation}
 Without loss of generality, we may assume that the sequence $\{y^{(k)}\}$ strongly converges to some vector $y\in \mathbb{R}^n\setminus\{0\}$ with $\|y\|=1$.
 Let $k\to\infty$ in (\ref{eq:45}). Then we have $\|y^+\|=0$, and so, $y\in \mathbb{R}^n_-\setminus\{0\}$ and $-y\in \mathbb{R}^n_+\setminus\{0\}$.  It follows from the strict copositivity of $\mathcal{A}$ that $$\mathcal{A}y^m=\mathcal{A}(-y)^m>0$$
 Since $\lim\limits_{k\to\infty}\mathcal{A}(y^{(k)})^m=\mathcal{A}y^m$, there exists a positive integer $N$ such that $$\mathcal{A}(y^{(k)})^m>0\mbox{ for all }k>N.$$
 This yields a contradiction of (\ref{eq:45}). Thus there is $\gamma\geq0$ such that the inequality (\ref{eq:42}) holds.

  Conversely, take $x\in \mathbb{R}^n_+\setminus\{0\}$. Obviously, $-x\in \mathbb{R}^n_-\setminus\{0\}$ and $\|(-x)^+\|=0$. Therefore, it follows from the inequality (\ref{eq:42}) that
 $$\mathcal{A}x^m=\mathcal{A}(-x)^m+\gamma\|(-x)^+\|^m>0\mbox{ for all }x\in \mathbb{R}^n_+\setminus\{0\}.$$
 Therefore, $\mathcal{A}$ is strictly copositive, as required.
 \end{proof}


 Using similar proof of Proposition \ref{pr:41}, we also easily prove the following conclusions.

 \begin{thm} \label{th:45} Let $\mathcal{A}$ be a symmetric tensor of order $m$ and dimension $n$. Then $\mathcal{A}$ is strictly copositive if and only if there is a real number $\gamma\geq0$ such that  \begin{equation} \label{eq:46}\mathcal{A}x^m+\gamma\|x^-\|^m>0, \mbox{ for all }x\in \mathbb{R}^n\mbox{ with }\|x\|=1.\end{equation}
 \end{thm}

  \begin{thm} \label{th:46} Let $\mathcal{A}$ be a symmetric tensor of order $m$ and dimension $n$. If  $m$ is an even number,  then $\mathcal{A}$ is strictly copositive if and only if there is a real number $\gamma\geq0$ such that  \begin{equation} \label{eq:47}\mathcal{A}x^m+\gamma\|x^+\|^m>0, \mbox{ for all }x\in \mathbb{R}^n\mbox{ with }\|x\|=1.\end{equation}
 \end{thm}

\section{\bf Principal Sub-tensors of Copositive  Tensors}

In homogeneous
polynomial $\mathcal{A}x^m$ defined by (\ref{eq:12}), if we let some (but not all) $x_i$ be zero, then we have a less variable homogeneous polynomial, which defines a lower dimensional tensor. We call such a lower dimensional tensor a {\em principal sub-tensor} of $\mathcal{A}$, i.e.,  an $m$-order $r$-dimensional principal sub-tensor $\mathcal{B}$  of an $m$-order $n$-dimensional tensor $\mathcal{A}$ consists of $r^m$ elements in $\mathcal{A} = (a_{i_1\cdots i_m})$: for any set $\mathcal{N}$ that composed of $r$ elements in $ \{1,2,\cdots , n\}$,
$$\mathcal{B} = (a_{i_1\cdots i_m}),\mbox{ for all } i_1, i_2, \cdots, i_m\in \mathcal{N}.$$ The concept were first introduced and used by Qi \cite{LQ1} for  the higher order symmetric tensor. Now we will continue to study the properties of the (strictly) copositive tensors by means of the principal sub-tensor of higher order symmetric tensor.

 \begin{thm} \label{th:32} Let $\mathcal{A}$ be a symmetric tensor of order $m$ and dimension $n$. Then $\mathcal{A}$ is copositive if and only if every principal sub-tensor of $\mathcal{A}$ has no  negative $H^{++}$-eigenvalue, i.e., every principal  sub-tensor of $\mathcal{A}$ has no eigenvector $v>0$ with
associated $H$-eigenvalue $\lambda<0$.
 \end{thm}

\begin{proof} Let $\mathcal{A}x^m\geq0$ for all $x\geq0$. Suppose there exists an $m$-order $l$-dimensional principal sub-tensor $\mathcal{B}$ of $\mathcal{A}$ with an $H^{++}$-eigenvalue $\lambda<0$ $(1\leq l\leq n)$, i.e. there is a positive vector $v\in\mathbb{R}^l_{++}$ such that $$\mathcal{B}v^{m-1}=\lambda v^{[m-1]}. $$
Without loss of generality, we may write
$v=(v_1,v_2,\cdots,v_l)^T$ $(v_i>0$ for $i=1,\cdots,l)$, $x_0=(v_1,v_2,\cdots,v_l,0,\cdots,0)^T$. Then by the definition of principal sub-tensor, we have $$\mathcal{A}x_0^m=\mathcal{B}v^m=\lambda \sum_{i=1}^lv_i^m<0.$$ This contradicts the hypothesis that $\mathcal{A}$ is copositive. So, no principal sub-tensor of $\mathcal{A}$ can have a negative $H^{++}$-eigenvalue.

Conversely, let the principal sub-tensor of $\mathcal{A}$ have the property said in the theorem. Suppose that $\mathcal{A}$ is not copositive, then at least there exists a $y_0\geq0$ such that $\mathcal{A}y_0^m<0$. Clearly, $y_0\neq0$.
Since the function $f(x)=Ax^m$ is continuous and the set $K=\{x \in \mathbb{R}^n_+;\sum\limits_{i=1}^nx_i^m = 1\}$ is a
compact subset of $\mathbb{R}^n$, there is $y\in K$ such that
\begin{equation} \label{eq:31}f(y)=\mathcal{A}y^m=\min_{x\in K} \mathcal{A}x^m.
\end{equation}
Obviously,   $f(y)=\mathcal{A}y^m\leq\mathcal{A}y_0^m<0$, and so some, but not all, components of $y$ may be $0$.
Without loss of generality, we may assume that $y_i>0$ for $i=1,2,\cdots,l$ $(1\leq l\leq n)$ and $y=(y_1,y_2,\cdots,y_l,0,\cdots,0)^T$,  and write $w=(y_1,y_2,\cdots,y_l)^T>0$. Let $\mathcal{B}$ be a principal sub-tensor that obtained from $\mathcal{A}$ by the polynomial $\mathcal{A}x^m$ for $x=(x_1,x_2,\cdots,x_l,0,\cdots,0)^T$. Then \begin{equation} \label{eq:32}
w\in\mathbb{R}^l_{++},\  \sum\limits_{i=1}^ly_i^m = 1\mbox{ and }f(y)=\mathcal{A}y^m=\mathcal{B}w^m<0.
\end{equation}

Let $x=(z_1,z_2,\cdots,z_l,0,\cdots,0)^T\in\mathbb{R}^n$ for all $z=(z_1,z_2,\cdots,z_l)^T\in\mathbb{R}^l$ with $\sum\limits_{i=1}^lz_i^m = 1$. Clearly, $x\in K$, and hence, by (\ref{eq:31}) and (\ref{eq:32}), we have $$f(x)=\mathcal{A}x^m=\mathcal{B}z^m\geq f(y)=\mathcal{A}y^m=\mathcal{B}w^m.$$ Since $w\in\mathbb{R}^l_{++}$,
$w$ is a local minimizer of the following optimization problem
$$ \begin{aligned}
    \min_{z\in \mathbb{R}^l} &\ \mathcal{B}z^m\\
     s.t. &\ \sum\limits_{i=1}^lz_i^m = 1.
      \end{aligned}$$
So, the standard KKT conditions
implies that there exists $\mu\in \mathbb{R}$ such that $$\nabla(\mathcal{B}z^m)-\mu\nabla(\sum\limits_{i=1}^lz_i^m -1)|_{z=w}=m\mathcal{B}w^{m-1}-m\mu w^{[m-1]}=0,$$
and hence $$\mathcal{B}w^{m-1}=\mu w^{[m-1]}\mbox{ and }\mathcal{B}w^m=\mu\sum\limits_{i=1}^ly_i^m =\mu<0.$$
This implies that the negative real number $\mu$ would be an $H^{++}$-eigenvalue of a principal sub-tensor $\mathcal{B}$ of $\mathcal{A}$. By hypothesis, it's out of the question. So, $\mathcal{A}x^m\geq0$ for all $x\geq0$, as required.
\end{proof}

Using the same proof as that of Theorem \ref{th:32} with appropriate changes in the
inequalities. We can obtain the following conclusions about the strictly copositive tensor.

 \begin{thm} \label{th:33} Let $\mathcal{A}$ be a symmetric tensor of order $m$ and dimension $n$. Then $\mathcal{A}$ is strictly copositive if and only if every principal sub-tensor of $\mathcal{A}$ has no non-positive $H^{++}$-eigenvalue, i.e., every principal  sub-tensor of $\mathcal{A}$ has no eigenvector $v>0$ with
associated $H$-eigenvalue $\lambda\leq0$.
 \end{thm}

 \begin{thm} \label{th:34} Let $\mathcal{A}$ be a symmetric tensor of order $m$ and dimension $n$. Then $\mathcal{A}$ is strictly copositive if and only if every principal sub-tensor of $\mathcal{A}$ has no non-positive $Z^{++}-$eigenvalue, i.e., every principal  sub-tensor of $\mathcal{A}$ has no eigenvector $v>0$ with
associated $Z$-eigenvalue $\lambda\leq0$.
 \end{thm}

 \begin{proof} Let $\mathcal{A}x^m>0$ for all $x\geq0$ with $x\neq0$. Suppose there exists an $m$-order $k$-dimensional principal sub-tensor $\mathcal{B}$ of $\mathcal{A}$ with an $Z^{++}$-eigenvalue $\lambda\leq0$ $(1\leq l\leq n)$, i.e. there is a positive vector $v\in\mathbb{R}^l_{++}$ such that
 $$\begin{cases}
 \mathcal{B}v^{m-1}=\lambda v\\
 v^Tv=1
 \end{cases}$$
We may write
$v=(v_1,v_2,\cdots,v_l)^T$ $(v_i>0$ for $i=1,\cdots,l)$. Let $x_0=(v_1,v_2,\cdots,v_l,0,\cdots,0)^T$. Then  we have $$\mathcal{A}x_0^m=\mathcal{B}v^m=\lambda \sum_{i=1}^lv_i^2=\lambda\leq 0.$$ This contradicts the hypothesis. So, no principal sub-tensor of $\mathcal{A}$ can have a negative or zero $Z^{++}$-eigenvalue.

Conversely, let each principal sub-tensor of $\mathcal{A}$ have no non-positive $Z^{++}$-eigenvalue. Suppose that $\mathcal{A}$ is not strictly copositive, then at least there exists a $y_0\geq0$ with $y_0\neq0$ such that $\mathcal{A}y_0^m\leq 0$.
Since the function $f(x)=Ax^m$ is continuous and the set $S=\{x \in \mathbb{R}^n_+;\sum\limits_{i=1}^nx_i^2 = 1\}$ is a
compact subset of $\mathbb{R}^n$, there is $y\in S$ such that
\begin{equation} \label{eq:33}f(y)=\mathcal{A}y^m=\min_{x\in S} \mathcal{A}x^m.
\end{equation}
Obviously, we must obtain that $f(y)=Ay^m\leq\mathcal{A}y_0^m\leq 0$.
Since $y_0\geq0$ with $y_0\neq0$, we may assume that $y=(y_1,y_2,\cdots,y_l,0,\cdots,0)^T$ $(y_i>0$ for $i=1,\cdots,l, 1\leq l\leq n)$. Let $w=(y_1,y_2,\cdots,y_l)^T$ and let $\mathcal{B}$ be a principal sub-tensor that obtained from $\mathcal{A}$ by the polynomial $\mathcal{A}x^m$ for $x=(x_1,x_2,\cdots,x_l,0,\cdots,0)^T$. Then \begin{equation} \label{eq:34}
w\in\mathbb{R}^l_{++},\  \sum\limits_{i=1}^ly_i^2 = 1\mbox{ and }f(y)=\mathcal{A}y^m=\mathcal{B}w^m\leq 0.
\end{equation}

Let $x=(z_1,z_2,\cdots,z_l,0,\cdots,0)^T\in\mathbb{R}^n$ for all $z=(z_1,z_2,\cdots,z_l)^T\in\mathbb{R}^l$ with $\sum\limits_{i=1}^lz_i^2 = 1$. Clearly, $x\in S$, and hence, by (\ref{eq:33}) and (\ref{eq:34}), we have $$f(x)=\mathcal{A}x^m=\mathcal{B}z^m\geq f(y)=\mathcal{A}y^m=\mathcal{B}w^m.$$ Since $w\in\mathbb{R}^l_{++}$,
$w$ is a local minimizer of the following optimization problem
$$ \begin{aligned}
    \min_{z\in \mathbb{R}^l} &\ \mathcal{B}z^m\\
     s.t. &\ \sum\limits_{i=1}^lz_i^2 = 1.
      \end{aligned}$$
So, the standard KKT conditions
implies that there exists $\mu\in \mathbb{R}$ such that
$\nabla(\mathcal{B}z^m-\mu(\sum\limits_{i=1}^lz_i^2 -1))|_{z=w}= m\mathcal{B}w^{m-1}-2\mu w =0,$
and hence $$\mathcal{B}w^{m-1}=\frac{2\mu}m w\mbox{ and }\mathcal{B}w^m=\frac{2\mu}m\sum\limits_{i=1}^ly_i^2 =\frac{2\mu}m\leq0.$$
This implies that the non-positive real number $\frac{2\mu}m$ would be an $Z^{++}$-eigenvalue of a principal sub-tensor $\mathcal{B}$ of $\mathcal{A}$. By hypothesis, it's quite impossible. So, $\mathcal{A}x^m>0$ for all $x\geq0$ with $x\neq0$, as required.
\end{proof}

Using the same proof as that of Theorem \ref{th:34} with appropriate changes in the
inequalities. We can obtain the following conclusions about the copositive tensor.

 \begin{thm} \label{th:35} Let $\mathcal{A}$ be a symmetric tensor of order $m$ and dimension $n$. Then $\mathcal{A}$ is copositive if and only if every principal sub-tensor of $\mathcal{A}$ has no  negative $Z^{++}$-eigenvalue, i.e., every principal  sub-tensor of $\mathcal{A}$ has no eigenvector $v>0$ with
associated $Z$-eigenvalue $\lambda<0$.
 \end{thm}

 Clearly, we have the following conclusions as a corollary of the above theorems.


  \begin{cor}\label{co:37} Let $\mathcal{A}$ be a symmetric tensor of order $m$ and dimension $n$. Then \begin{itemize}
\item[(i)] $\mathcal{A}$ is copositive if and only if  for every principal sub-tensor $\mathcal{B}$ of $\mathcal{A}$, the fact that $\lambda$ is $H^{++}$ (or $Z^{++}$)-eigenvalue of $\mathcal{B}$  means that  $\lambda\geq0$;
\item[(ii)] $\mathcal{A}$ is strictly copositive if and only if for every principal sub-tensor $\mathcal{B}$ of $\mathcal{A}$, the fact that $\lambda$ is $H^{++}$ (or $Z^{++}$)-eigenvalue of $\mathcal{B}$  means that  $\lambda>0$.\end{itemize}
 \end{cor}

\begin{cor} \label{co:38} Let $\mathcal{A}$ be a symmetric tensor of order $m$ and dimension $n$.
\begin{itemize}
\item[(i)] If $\mathcal{A}$ is copositive, then $a_{ii\cdots i}\geq0$ for all $i=1,2,\cdots,n$.
\item[(ii)] If $\mathcal{A}$ is strictly copositive, then $a_{ii\cdots i}>0$ for all $i=1,2,\cdots,n$.
\item[(iii)] Let $a_{ii\cdots i}=0$ for $i\in\{1,2,\cdots,n\}$. If $\mathcal{A}$ is copositive, then $a_{ii\cdots ij}\geq0$ for all $j=1,2,\cdots,n$. 
\end{itemize} \end{cor}

\begin{proof} (i) For each fixed $i$, $\mathcal{B}_i=(a_{ii\cdots i})$ is an $m$-order $1$-dimensional principal sub-tensor $\mathcal{B}$ of $\mathcal{A}$. Clearly, $a_{ii\cdots i}$ is the unique eigenvalue with eigenvector $(1)$. It follows from Theorem \ref{th:32} (or \ref{th:35}) that $\mathcal{B}_i$ has no a negative $H^{++}$ ( or $Z^{++}$)-eigenvalue.  Thus $a_{ii\cdots i}\geq0$ for each $i\in\{1,2,\cdots,n\}$.

Similarly, from Theorem \ref{th:33} (or \ref{th:34}), (ii)  is easily obtained.

(iii) Let $\mathcal{B}_{ij}$ be a $m$-order $2$-dimensional principal sub-tensor of $\mathcal{A}$. Then by the definition of the principal sub-tensor, we have $$\mathcal{B}_{ij}=(a_{i_1i_2\cdots i_k}),\ \ \ i_1,i_2,\cdots, i_k=i,j.$$ 
Let $x=(t,1)^T$ and $y=(t,1,0,0,\cdots,0)^T$ for $t>0$. Then
$$\begin{aligned}
0\leq\mathcal{A}y^m=\mathcal{B}_{ij}x^m=&t^ma_{ii\cdots i}+mt^{m-1}a_{ii\cdots ij}+\cdots+a_{jj\cdots jj}\\
=&mt^{m-1}a_{ii\cdots ij}+\left(\begin{aligned}2\\m\end{aligned}\right)t^{m-2}a_{ii\cdots ijj}+\cdots+a_{jj\cdots jj}.
\end{aligned}$$
So we have
$$mt^{m-1}a_{ii\cdots ij}\geq -\left(\begin{aligned}2\\m\end{aligned}\right)t^{m-2}a_{ii\cdots ijj}-\cdots-a_{jj\cdots jj},$$
and hence
$$a_{ii\cdots ij}\geq -\left(\begin{aligned}2\\m\end{aligned}\right)\frac{a_{ii\cdots ijj}}{mt}-\cdots-\frac{a_{jj\cdots jj}}{mt^{m-1}}.$$
Let $t\to \infty$. We have $a_{ii\cdots ij}\geq0.$ \end{proof}


 It follows from the above results that the copositivity of tensors implies that all diagonal entry can not be negative, and the strict copositivity implies that all diagonal entry must be positive. So this reveals that in testing for (strict) copositivity of tensors, one can restrict attention to tensors with  (positive) non-negative diagonal entries.  We can further restrict attention to tensors whose diagonal entries are all equal to $1$; we say that such a tensor has unit diagonal.

 \begin{ex} \label{ex:1} Let $\mathcal{A}$ be a symmetric tensor of order $3$ and dimension $2$.  If $a_{111}=a_{222}=0$, then it follows from Corollary \ref{co:38} (iii) that $a_{iij}\geq0$ for $i,j=1,2$ whenever $\mathcal{A}$ is copositive. So $\mathcal{A}$ is copositive if and only if $\mathcal{A}$ is non-negative.

 If $\mathcal{A}$ is copositive   and   $a_{111}=a_{222}=1$, then $a_{111}+a_{222}+3a_{112}+3a_{221}\geq0$, and hence  $$ a_{112}+a_{221}\geq-\frac23.$$
\end{ex}

 \begin{ex} \label{ex:2} Let $\mathcal{A}$ be a copositive and symmetric tensor of order $3$ and dimension $3$.  If $a_{111}=a_{222}=a_{333}=0$, then $$a_{112}\geq0,a_{113}\geq0, a_{221}\geq0, a_{331}\geq0, a_{332}\geq0, a_{223}\geq0$$ and  $$2a_{123}+a_{112}+a_{113}+a_{221}+a_{223}+a_{331}+a_{332}\geq0.$$ If   $a_{111}=a_{222}=a_{333}=1$, then consider an $3$-order $2$-dimensional principal sub-tensor of $\mathcal{A}$, we have $$ a_{112}+a_{221}\geq-\frac23, a_{113}+a_{331}\geq-\frac23, a_{332}+a_{223}\geq-\frac23,$$ and  $$2a_{123}+a_{112}+a_{113}+a_{221}+a_{223}+a_{331}+a_{332}\geq-1.$$
\end{ex}

\section*{Acknowledgments}
 The authors would like to express
their sincere thanks to Dr. Lek-Heng Lim, the Handling Editor and
the anonymous referees for their constructive comments and valuable
suggestions.
\bibliographystyle{amsplain}

\end{document}